\documentclass[11pt]{amsart}
\usepackage{amsmath,amssymb,mathrsfs}
\usepackage[colorlinks=true,linkcolor=black,citecolor=black,urlcolor=black]{hyperref}

\newcommand{\define}{\textbf}

\renewcommand{\setminus}{\smallsetminus}
\renewcommand{\phi}{\varphi}

\renewcommand{\tilde}{\widetilde}
\renewcommand{\hat}{\widehat}

\newcommand{\Q}{\mathbb{Q}}

\newcommand{\Z}{\mathbb{Z}}

\newcommand{\cL}{\mathcal{L}}

\newcommand{\bdy}{\partial}
\newcommand{\tbdy}{\tilde\partial}
\newcommand{\Bdy}{\Delta}
\newcommand{\tBdy}{\tilde\Delta}
\newcommand{\tc}{\tilde{c}}

\DeclareMathOperator{\hgt}{ht}

\newtheorem*{theorem}{Theorem}
\newtheorem*{lemma}{Lemma}

\theoremstyle{definition}

\newtheorem*{remark}{Remark}

\begin{document}

\title{Schubert varieties are log Fano over the integers}
\author{Dave Anderson}
\address{Department of Mathematics\\University of Washington\\Seattle, WA 98195}
\address{Department of Mathematics\\University of British Columbia\\ BC, Canada V6T 1Z2}
\email{anderson@math.ubc.ca, dandersn@math.washington.edu}
\author{Alan Stapledon}
\email{astapldn@math.ubc.ca}
\date{March 26, 2012}
\thanks{DA was partially supported by NSF Grant DMS-0902967.}

\begin{abstract}
Given a Schubert variety $X_w$, we exhibit a divisor $\Delta$, defined over $\Z$, such that the pair $(X_w,\Delta)$ is log Fano in all characteristics.
\end{abstract}

\maketitle

Let $X_w=\overline{BwB/B}$ be a Schubert variety in a Kac-Moody flag variety $G/B$ over an algebraically closed field of arbitrary characteristic.  Fix a reduced word for the Weyl group element $w$, and let $\phi:\tilde{X}_w \to X_w$ be the corresponding Bott-Samelson resolution.  Let $\bdy = \bdy_1+\cdots+\bdy_k$ be the complement of the open $B$-orbit in $X_w$, written as a sum of prime divisors, and define $\tbdy=\tbdy_1+\cdots+\tbdy_\ell$ similarly; the divisor $\tbdy$ is a simple normal crossings divisor.  Here $\ell=\ell(w)$ is equal to the length of $w$, which in turn is equal to the dimension of $X_w$.

Denote by $\rho$ the sum of all fundamental weights of (the simply connected form of) $G$, and let $\cL(\rho)$ be the corresponding ample line bundle on $X_w$.  Choosing $B$-invariant sections of $\cL(\rho)$ and $\phi^*\cL(\rho)$, these line bundles correspond to divisors
\[
  a_1\bdy_1 + \cdots + a_k\bdy_k, \qquad \text{ and } \qquad b_1 \tbdy_1 + \cdots + b_\ell \tbdy_\ell,
\]
for some nonnegative integers $a_i$ and $b_i$.  In fact, they are all positive, by the Lemma below.

Recall that for a normal irreducible variety $Y$ and an effective $\Q$-divisor $D$, the pair $(Y,D)$ is \define{Kawamata log terminal (klt)} if $K_Y + D$ is $\Q$-Cartier, and for all proper birational maps $\pi:Y' \to Y$, the pullback $\pi^*(K_Y + D) = K_{Y'} + D'$ has $\pi_*K_{Y'} = K_Y$ and $\lfloor D' \rfloor \leq 0$.  The pair is \define{log Fano} if it is klt and $-(K_Y+D)$ is ample.

In \cite{ss}, Schwede and Smith prove that globally $F$-regular varieties are log Fano, for some choice of boundary divisor that may depend on the characteristic.  Motivated by the fact that Schubert varieties are known to be globally $F$-regular \cite{lrt}, they ask if there exists a single boundary divisor that works uniformly, in all characteristics.  The purpose of this note is to answer their question affirmatively:

\begin{theorem}
The pair $(X_w,\Bdy)$ is log Fano, where $\Bdy = c_1 \bdy_1 + \cdots + c_k\bdy_k$, and $c_i=1-a_i/M$ for some integer $M$ greater than all $a_i$.
\end{theorem}

We will need a lemma.

\begin{lemma}
For all $i$, we have $b_i>0$.
\end{lemma}

\begin{proof}
Let $(\alpha_1,\ldots,\alpha_\ell)$ be the sequence of simple roots corresponding to the reduced word defining $\tilde{X}_w$, and let $s_1,\ldots,s_\ell$ be the corresponding simple reflections.  For each $i$, set
\begin{align}\label{eq:vi}
  v_i &= s_1\cdots \hat{s}_i \cdots s_\ell = ws_{\gamma_i}, \\
\intertext{where}
\label{eq:gamma}
  \gamma_i &= s_\ell s_{\ell-1}\cdots s_{i+1}(\alpha_i).
\end{align}
We claim that
\begin{equation}\label{eq:bi}
  b_i = \langle \rho, \gamma_i^\vee \rangle = \hgt(\gamma_i),
\end{equation}
where $\hgt(\beta) = |\sum n_i|$ if the root $\beta = \sum n_i \alpha_i$ as a sum of simple roots.  (Since the expression $w=s_1\cdots s_\ell$ is \emph{reduced}, each $\gamma_i$ is a positive root.)

The claimed expression follows from Chevalley's formula for multiplying a Schubert class by a divisor \cite{chevalley}, but it is easy enough to prove directly.  We first review some basic geometry of the Bott-Samelson variety.  
We have $\tilde{X}_w = (P_{\alpha_1}\times \cdots \times P_{\alpha_\ell})/B^\ell$, where $P_{\alpha}$ is the minimal parabolic subgroup, and $B^\ell$ acts by
\[
  (p_1,p_2,\ldots,p_\ell)\cdot(b_1,b_2,\ldots,b_\ell)=(p_1b_1,b_1^{-1}p_2b_2,\ldots,b_{\ell-1}^{-1}p_\ell b_\ell).
\]
For each $i$, fix a representative $\dot{s}_i \in P_{\alpha_i}$.  The torus $T$ acts on $\tilde{X}_w$ by left multiplication, and its fixed points correspond to subsets $I\subseteq [\ell]:=\{1,\ldots,\ell\}$.  Specifically, the point $p(I)$ is defined by $p_i=\dot{s}_i$ if $i\in I$ and $p_i=e$ otherwise.  The divisor $\tbdy_i$ is defined by the equation $p_i = e$.

For each $i=1,\ldots,\ell$, one can find a $T$-invariant curve $C_i$ in $\tilde{X}_w$, whose fixed points are $p([\ell])$ and $p([\ell]\setminus\{i\})$.  By considering $T$-fixed points, it is easy to see that $C_i\cap \tbdy_i = p([\ell]\setminus\{i\})$, and $C_i\cap\tbdy_j = \emptyset$ if $i\neq j$.  It follows that $b_i = (\phi^*\cL(\rho)\cdot C_i)$.

By the projection formula, this intersection number is equal to $(\cL(\rho)\cdot \phi(C_i))$.  The curve $\phi(C_i)$ is the unique $T$-invariant curve in $X_w\subseteq G/B$ with fixed points $wB$ and $v_iB$.  The line bundle $\cL(\rho)$ on $G/B$ is invariant under the action of the Weyl group, so we may apply a translation by $w^{-1}$; now $w^{-1}\phi(C_i)$ is the curve with fixed points $eB$ and $w^{-1}v_iB = s_{\gamma_i}B$.  The formula \eqref{eq:bi} follows, since for any dominant weight $\lambda$ and any positive root $\beta$, if $C$ is the $T$-invariant curve in $G/B$ joining $eB$ and $s_\beta B$, we have $(\cL(\lambda)\cdot C) = \langle \lambda,\beta^\vee \rangle$.
\end{proof}

\begin{proof}[Proof of Theorem]
The canonical divisors for $X_w$ and $\tilde X_w$ are well known (see, e.g., \cite{ramanathan}).  We have
\begin{align*}
  K_{X_w} &= -(a_1+1)\bdy_1 - \cdots - (a_k+1)\bdy_k \\
\intertext{and}  
K_{\tilde X_w} &= -(b_1+1)\tbdy_1 - \cdots - (b_\ell+1)\tbdy_\ell.
\end{align*}

By the definition of $\Bdy$, the (integral) divisor
\begin{align*}
  MK_{X_w} + M\Bdy &= (-Ma_1-M + M-a_1)\bdy_1 + \cdots + (-Ma_1-M + M-a_1)\bdy_k \\
                   &= -(M+1)a_1\bdy_1 - \cdots - (M+1)a_k\bdy_k
\end{align*}
comes from a section of $\cL(-\rho)^{\otimes M+1}$, so $K_{X_w}+\Bdy$ is $\Q$-Cartier and anti-ample.

Now set $\tc_i = 1-b_i/M$, and let $\tBdy=\sum \tc_i\tbdy_i$.  From the definition, we have
\begin{align*}
  \phi^*(K_{X_w}+\Bdy) &= \frac{M+1}{M}(-b_1 \tbdy_1 - \cdots - b_\ell \tbdy_\ell) \\
                       &= (-b_1-1+\frac{1}{M}(M-b_1))\tbdy_1 + \cdots + (-b_\ell-1+\frac{1}{M}(M-b_\ell))\tbdy_\ell \\
                       &= K_{\tilde{X}_w} + \tBdy.
\end{align*}               
The Lemma implies that $\lfloor \tBdy \rfloor \leq 0$.

Next we show that $\phi_*\tBdy=\tBdy$.  Since $X_w$ is normal, for each $j=1,\ldots,k$, there is a unique $i(j)$ such that $\phi(\tbdy_{i(j)})=\bdy_j$; the remaining $\tbdy_i$'s are collapsed by $\phi$.  It follows that for any divisor $\tilde{D}=\sum_i d_i\,\tbdy_i$, we have $\phi_*D = \sum_j d_{i(j)}\,\bdy_j$.  For the same reason, $\phi_*\phi^*\cL(\rho)=\cL(\rho)$, so $a_j = b_{i(j)}$.  The claim follows.

Finally, the map $\phi:\tilde X_w \to X_w$ is a rational resolution (see \cite{ramanathan}), so in particular we have $\phi_*K_{\tilde X_w} = K_{X_w}$.  Applying \cite[Lemma~2.30 and Corollary 2.31]{km}, it follows that $(X_w,\Delta)$ is klt, and the Theorem is proved.
\end{proof}

\begin{remark}
The formula
\[
  -K_{\tilde{X}_w} = \sum_i (\langle \rho, \gamma_i^\vee \rangle + 1)\, \tbdy_i
\]
proved in the Lemma can be found in \cite[Proof of Proposition~10]{mr}, and is also valid for Bott-Samelson varieties corresponding to non-reduced words.  In this generality the coefficients on the right-hand side may be negative, since some of the $\gamma_i$'s will be negative roots.  Note, however, that the anticanonical divisor is always effective.
\end{remark}

\medskip
\noindent
{\em Acknowledgements.}  We thank Karl Schwede and Karen Smith for asking us this question, as well as Shinnosuke Okawa and the referee for helpful comments.



\end{document}